\documentclass[draft]{amsart}

\newtheorem{theorem}{Theorem}
\newtheorem{lemma}[theorem]{Lemma}
\newtheorem{proposition}[theorem]{Proposition}
\newtheorem{cor}[theorem]{Corollary}
\renewcommand{\Re}{\operatorname{Re}}
\newcommand{\meas}[1]{\frac{1}{T}\operatorname{meas}\Bigl\{\tau\in [0,T]: #1 \Bigr\}}

\begin{document}
\title{Joint universality for Lerch zeta-functions}

\author{Yoonbok Lee}
\address{1 Department of Mathematics, 2 Research Institute of Natural Sciences, Incheon National University, Incheon, Korea}
\email{leeyb131@gmail.com, leeyb@inu.ac.kr}

\author{Takashi Nakamura}
\address{Department of Liberal Arts, Faculty of Science and Technology, Tokyo University of Science, 2641 Yamazaki, Noda-shi, Chiba-ken, 278-8510, Japan}
\email{nakamuratakashi@rs.tus.ac.jp}

\author{{\L}ukasz Pa\'nkowski}
\address{Faculty of Mathematics and Computer Science, Adam Mickiewicz University, Umultowska 87, 61-614 Pozna\'{n}, Poland, and Graduate School of Mathematics, Nagoya University, Nagoya, 464-8602, Japan}
\email{lpan@amu.edu.pl}

\thanks{The second author was partially supported by JSPS Grant no.~24740029. The third author was partially supported by (JSPS) KAKENHI grant no.~26004317 and the grant no.~2013/11/B/ST1/02799 from the National Science Centre.}

\subjclass[2010]{Primary: 11M35}
\keywords{Joint universality, Lerch zeta-functions}

\begin{abstract}
For $0<\alpha, \lambda \leq 1$, the Lerch zeta-function is defined by $L(s;\alpha, \lambda)$$:= \sum_{n=0}^\infty e^{2\pi i\lambda n} (n+\alpha)^{-s}$, where $\sigma>1$. In this paper, we prove joint universality for Lerch zeta-functions with distinct $\lambda_1,\ldots,\lambda_m$ and transcendental $\alpha$. 
\end{abstract}

\maketitle

\section{Introduction and statement of main result}
For $0 < \alpha, \lambda \le 1$, we define the Lerch zeta-function by 
\[
L(s;\alpha, \lambda) := \sum_{n=0}^\infty\frac{e(\lambda n)}{(n+\alpha)^s},\qquad \sigma>1,
\]
where $e(t) = \exp(2\pi i t)$. When $\lambda = 1$, the function $L(s;\alpha,\lambda)$ reduces to the Hurwitz zeta-function $\zeta(s,a)$. If $\lambda \ne 1$, the Lerch zeta-function $L(s;\alpha, \lambda)$ is analytically continuable to an entire function. However, the Hurwitz zeta-function $\zeta(s,a)$ is extended to a meromorphic function, which has a simple pole at $s = 1$. 

In this paper, we show the following joint universality theorem expected by Mishou \cite[Conjecture 1]{MLerch}. In order to state it, put $D:= \{s \in \mathbb{C} : 1/2<\Re s< 1\}$ and let ${\rm{meas}} \{A\}$ be the Lebesgue measure on ${\mathbb{R}}$ of the set $A$.

\begin{theorem}\label{th:main}
Suppose that $L(s;\alpha,\lambda_1),\ldots,L(s;\alpha,\lambda_m)$ are Lerch zeta-functions with distinct $\lambda_1,\ldots,\lambda_m$ and transcendental $\alpha$. For $1 \le j \le m$, let $K_j \subset D$ be compact sets with connected complements and $f_j (s)$ be continuous function on $K_j$ and analytic in the interior of $K_j$. 
Then, for every $\varepsilon>0$, we have
$$
\liminf_{T\to\infty} \meas{ \max_{1 \le j \le m}\max_{s\in K_j} 
\bigl|L(s+i\tau;\alpha,\lambda_j)-f_j(s)\bigr| <\varepsilon }>0.
$$
\end{theorem}
Roughly speaking, this theorem implies that any analytic functions can be simultaneously and uniformly approximated by Lerch zeta-functions with distinct $\lambda_1,\ldots,\lambda_m$. The proof will be written in Sections 2 and 3. We skip the detail of the proofs of results appeared in Section 2 since they do not contain essentially new ideas. In Section 3, we prove the denseness lemma using an orthogonality of Dirichlet coefficients of the zeta-functions. The main idea of our proof was recently observed in \cite{LeeNakPanS} by the authors. However, in the present paper we adopt this approach to completely different kind of zeta-functions without Euler product. It proves the conjecture on joint universality for Lerch zeta-functions put forward by Mishou in \cite{MLerch} and shows that this idea can be applicable to many collections of zeta and $L$-functions, which independence relies on some orthogonality property of their coefficients.

Now we look back in the history of the joint universality for Lerch zeta-functions. Laurin\v{c}ikas showed Theorem \ref{th:main} with $m=1$ in \cite[Theorem]{Lau97} (see also \cite[Theorem 6.1.1]{LauGa}). Laurin\v{c}ikas and Matsumoto proved Theorem \ref{th:main} with the condition that $\lambda_j = k_j / l_j$ are distinct rational numbers satisfying $(k_j,l_j)=1$ and $0 < k_j \le l_j$ in \cite[Theorem 1]{LauMa} (see also \cite[Theorem 6.3.1]{LauGa} or \cite[Theorem 2]{MLerch}). In \cite[Theorem 17]{Na07}, Nakamura obtained the joint universality of the Lerch zeta-functions with $\lambda_j = \lambda + k_j / l_j$, where $0 < \lambda \le 1$ and $\lambda_j$ are distinct in $\rm{mod} \,\, 1$. We have to mention that the method in the both papers \cite[Theorem 1]{LauMa} and \cite[Theorem 17]{Na07} are based on the observation that if $e(\lambda_ln) \ne e(\lambda_kn)$, there exists $M \in \mathbb{N}$ such that for all $1 \le k \ne l \le m$,
$$
|e(\lambda_ln) - e(\lambda_kn)| = |1 - e((\lambda_l-\lambda_k)n)| \ge |1-e(1/M)|. 
$$ 
Recently, Mishou proved in \cite[Theorem 4]{MLerch}, the joint universality of the Lerch zeta-functions for almost all real numbers $\lambda_j$, $1 \le j  \le m$ such that $1, \lambda_1,\ldots ,\lambda_m$ are linearly independent over ${\mathbb{Q}}$. His proof is based on some results of discrepancy estimate from uniform distribution theory (see \cite[Section 2]{MLerch}). Obviously, Theorem \ref{th:main} of the present paper is not only an improvement of Mishou's result \cite[Theorem 4]{MLerch} but also the final answer to \cite[Conjecture 1]{MLerch}. 

By using Theorem \ref{th:main}, we get the following corollaries. We omit their proofs since they follow from the standard argument (see for example \cite[Section 7.2]{LauGa}).
\begin{cor}\label{cor:1}
Let $\alpha\in (0,1]$ be transcendental and $\lambda_1,\ldots,\lambda_m\in (0,1]$ be distinct real numbers. For $N\in\mathbb{N}$ and $1/2 < \sigma <1$, define the mapping $h \colon \mathbb{R} \to \mathbb{C}^{mN}$ by the formula
\begin{align*}
h(t) := \bigl( & L(\sigma+i\tau;\alpha,\lambda_1), L'(\sigma+i\tau;\alpha,\lambda_1), \ldots, 
L^{(N-1)}(\sigma+i\tau;\alpha,\lambda_1), \\ & ,\ldots, L(\sigma+i\tau;\alpha,\lambda_m), L'(\sigma+i\tau;\alpha,\lambda_m), \ldots, L^{(N-1)}(\sigma+i\tau;\alpha,\lambda_m) \bigr) .
\end{align*}
Then the image of $\mathbb{R}$ is dense in $\mathbb{C}^{mN}$. 
\end{cor}
\begin{cor}\label{cor:2}
Let $\alpha\in (0,1]$ be transcendental and $\lambda_1,\ldots,\lambda_m\in (0,1]$ be distinct real numbers. Suppose $N\in\mathbb{N}$ and $F_l$, $0 \le l \le k$ are continuous functions on $\mathbb{C}^{mN}$ and satisfy
\begin{align*}
\sum_{l=0}^k s^l F_l \bigl( & L(s;\alpha,\lambda_1), L'(s;\alpha,\lambda_1), \ldots, 
L^{(N-1)}(s;\alpha,\lambda_1) , \\ & ,\ldots, L(s;\alpha,\lambda_m), L'(s;\alpha,\lambda_m), \ldots, 
L^{(N-1)}(s;\alpha,\lambda_m) \bigr) \equiv 0.
\end{align*}
Then we have $F_l \equiv 0$ for $0 \le l \le k$. 
\end{cor}

\section{Proof of Theorem \ref{th:main}}
Recall that $D:=\{s\in\mathbb{C}:1/2<\Re s< 1\}$ and denote by $H(D)$ the space of analytic function on $D$ equipped with the topology of uniform convergence on compacta. Let ${\mathfrak{B}}(X)$ stand for the class of Borel sets of the space $X$. Define by $\gamma$ the unit circle on ${\mathbb{C}}$, and let $\Omega := \prod_{n=0}^{\infty} \gamma_n$, where $\gamma_n = \gamma$ for all $n \in  {\mathbb{N}}_0$. Denoting by $m_{H}$ the probability Haar measure on $(\Omega , {\mathfrak{B}}(\Omega))$, we obtain a probability space $(\Omega , {\mathfrak{B}}(\Omega) , m_{H})$. For $\sigma >1$, we define 
\[
L(s;\alpha, \lambda ; \omega) := 
\sum_{n=0}^\infty\frac{e(\lambda n) \omega (n)}{(n+\alpha)^s},\qquad \omega (n) \in \gamma. 
\]
Note that for almost all $\omega \in \Omega$ the series above converges uniformly on compact subsets of $D$ (see for instance \cite[Lemma 5.2.1]{LauGa}).

Let $H(D)^m:= H(D) \times \cdots \times H(D)$. We define a probability measure $P_T$ on $(H(D)^m, {\mathfrak{B}}(H(D)^m))$ by
$$
P_T (A) := \meas{ \bigl( L(s+i\tau;\alpha,\lambda_1), \ldots, L(s+i\tau;\alpha,\lambda_m)
\bigr) \in A }, 
$$
where $A \in {\mathfrak{B}}(H(D)^m)$. Next define the $H(D)^m$-valued random element $\underline{L}(s;\omega)$ by
$$
\underline{L}(s;\omega) := \bigl( L(s;\alpha, \lambda_1 ; \omega), \ldots ,
L(s;\alpha, \lambda_m ; \omega) \bigr).
$$
Denote by $P_{\underline{L}}$ the distribution of the random element $\underline{L}(s;\omega)$, namely, 
$$
P_{\underline{L}} (A) := m_H \bigl\{ \omega \in \Omega : \underline{L}(s;\omega) \in A \bigr\}, \qquad
A \in {\mathfrak{B}}(H(D)^m). 
$$
Then we have the following limit theorem proved by Matsumoto and Laurin\v{c}ikas \cite{LauMa} (see also \cite[Theorem 5.3.1]{LauGa} or \cite[Section 5]{MLerch}). 
\begin{proposition}[{\cite[Lemma 1]{LauMa}}]\label{pro:lim1}
Let $0<\alpha <1$ be transcendental. Then the probability measure $P_T$ converges weakly to $P_{\underline{L}}$ as $T \to \infty$. 
\end{proposition}
The proof of the next lemma shall be written in Section 3 since it contains the most novel part of the present paper.
\begin{lemma}\label{lem:den1}
The set $\{ \underline{L}(s;\omega) : \omega \in \Omega \}$ is dense in $H(D)^m$. 
\end{lemma}
Recall that the minimal closed set $S_{\bf{P}} \subset X$ such that ${\bf{P}}(S_{\bf{P}}) = 1$ is called the support of a probability space $(X,{\mathfrak{B}}(X),{\bf{P}})$. The set $S_{\bf{P}}$ consists of all $x \in S$ such that for every neighborhood $V$ of $x$ the inequality ${\bf{P}} (V) > 0$ is satisfied. From Lemma \ref{lem:den1} and \cite[Lemma 6.1.3]{LauGa} or \cite[Lemma 12.7]{S}, the support of the probability measure $P_{\underline{L}}$ is $H(D)^m$. First assume that $h_1(s), \ldots ,h_m(s) \in H(D)$ are polynomials. Let $K_j$  be the same as in Theorem \ref{th:main} and $\Phi$ be the set of functions $\underline{\varphi} \in H(D)^m$ which satisfy
$$
\max_{1 \le j \le m} \max_{s \in K_j} \bigl| \varphi_j (s) - h_j (s)\bigr| < \varepsilon.
$$
From Proposition \ref{pro:lim1}, the definition of support, Portmanteau theorem (see for instance \cite[Theorem 3.1]{S}) and the fact that the support of $P_{\underline{L}}$ is $H(D)^m$, we have
$$
\liminf_{T \to \infty} P_T (\Phi) \ge P_{\underline{L}} (\Phi) >0.
$$
Therefore, we obtain 
$$
\liminf_{T\to\infty} \meas{ \max_{1 \le j \le m}\max_{s\in K_j} 
\bigl|L(s+i\tau;\alpha,\lambda_j)-h_j(s)\bigr| <\varepsilon }>0.
$$
Hence it suffices to show that polynomials $h_j(s)$ can be replaced by $f_j(s)$ appeared in Theorem \ref{th:main}. It is possible by Mergelyan's theorem which implies that any function $f(s)$ which is continuous on $K$ and analytic in the interior of $K$, where $K$ is a compact subset with connected complement, is uniformly approximative on $K$ by polynomials. Hence we omit the details since this is easily done by the well-known method (see for example \cite[p.~129]{LauGa} or \cite[p.~1125]{MLerch}).
 
\section{Proof of Lemma \ref{lem:den1}}

Let $U$ be a simply connected smooth Jordan domain such that $\overline{U}\subset D$. Let $B^2(U)$ be the Bergman space of all holomorphic square integrable complex functions with respect to the Lebesgue measure on $U$ with the inner product
\[
\langle f,g\rangle = \iint_U f(s)\overline{g(s)} d\sigma dt, \qquad f,g \in H(U).
\]
The properties below are well-known (see for instance \cite{Que}).
\begin{lemma}[{\cite[Proposition 7.2.2 and Theorem 7.2.3]{Que}}]\label{lem:Ber}
We have the following.\\
$(a)$ Convergence in $B^2 (U)$ implies local uniform convergence on $U$.\\
$(b)$ $B^2(U)$ is a Hilbert space.\\
$(c)$ The set of polynomials is dense in $B^2(U)$. 
\end{lemma}
Now let ${\mathbb{B}}^m := B^2(U) \times \cdots \times B^2(U)$ is the Hilbert space with the inner product given, for $\underline{f} = (f_1,\ldots,f_m) \in H(U)^m$ and $\underline{g} = (g_1,\ldots,g_m) \in H(U)^m$ by
\[
\bigl\langle \underline{f},\underline{g} \bigr\rangle = 
\sum_{j=1}^m\iint_U f_j(s)\overline{g_j(s)} d\sigma dt .
\]

In order to prove Lemma \ref{lem:den1}, we use $(b)$ of Lemma \ref{lem:Ber} and the following result appeared, for example, in \cite{S}. 
\begin{lemma}[{\cite[Theorem 6.1.16]{S}}]\label{lem:complexHilbert}
Let $H$ be a complex Hilbert space. Assume that a sequence $v_n\in H$, $n\in\mathbb{N}$ satisfies\\
$\,(i)$ the series $\sum_n  \|v_n\|^2 < \infty$;\\
$(ii)$ for any element $0\ne g \in H$ the series $\sum_n |\langle v_n, g \rangle|$ is divergent.\\
Then the set of convergent series
\[
\left\{\sum_n a_nv_n\in H:|a_n|=1\right\}
\]
is dense in $H$.
\end{lemma}

Let $\underline{g}=(g_1,\ldots,g_m)\in {\mathbb{B}}^m$ be a non-zero element and put
$$
\underline{v_n} (s) := 
\bigl( v_n(s;\alpha,\lambda_1),\ldots, v_n (s;\alpha,\lambda_m) \bigr), \qquad
v_n (s;\alpha,\lambda_j) := \frac{e(\lambda_j n)}{(n+\alpha)^{s}}.
$$ 
Then for $\Delta_j(w) := \iint_U e^{-sw}\overline{g_j(s)}d\sigma dt$, one has
\[
\bigl\langle \underline{v_n}(s),\underline{g}(s) \bigr\rangle = 
\sum_{j=1}^m e(\lambda_j n)\Delta_j(\log(n+\alpha)).
\]
We can see that the condition (i) of Lemma \ref{lem:complexHilbert} is true since $\overline{U} \subset D$ and
$$
\bigl\langle \underline{v_n}(s),\underline{v_n}(s) \bigr\rangle = 
\sum_{j=1}^m \iint_U (n+\alpha)^{-s} \overline{(n+\alpha)^{-s}} d\sigma dt \ll
\sup_{s \in U} \Bigl| (n+\alpha)^{-2s} \Bigr|.
$$
The truth of the condition (ii) in Lemma \ref{lem:complexHilbert} easily follows from the following crucial lemma.

\begin{lemma}\label{lem:main}
Assume that $\underline{g}(s) = (g_1(s),\ldots,g_m(s))\in {\mathbb{B}}^m$ is a non-zero element and for $j=1,\ldots m$, put $\Delta_j(z) :=\iint_U e^{-sz}\overline{g_j(s)}d\sigma dt$. Then the following series
\[
\sum_{n=0}^\infty 
\bigl| e(\lambda_1 n)\Delta_1(\log(n+\alpha))+\cdots+e(\lambda_m n)\Delta_m(\log(n+\alpha)) \bigr|
\]
is divergent.
\end{lemma}
In order to prove the lemma above, we quote the following. 
\begin{lemma}[{\cite[Corollary 2.7]{LeeNakPanS}}]\label{cor:Delta}
Let $\|g_j \| \ne 0$ for $1 \le j \le m$. Then for every $A>0$ and every $x>1$, there exist sequences $B_1>\cdots>B_m>0$, $x^{(0)}_0=x, x^{(1)}_0,\ldots, x^{(m)}_0$ and intervals $I_j\subset[x,x+1]$ of length $|I_j|\geq B_jx^{-2j}$ such that $x^{(j)}_0\in I_j$, $I_{j+1}\subset I_j$, and for all $t\in I_{j}$ we have
\begin{equation}\label{eq:delxj}
\begin{split}
&\frac{1}{2} \bigl|\Delta_j(x^{(j-1)}_0)\bigr|+O\left(e^{-Ax}\right)\leq 
\frac{1}{2} \bigl|\Delta_j(x^{(j)}_0)\bigr|+O\left(e^{-Ax}\right) \\ &\leq 
\bigl|\Delta_j(t)\bigr|\leq \bigl|\Delta_j(x^{(j)}_0)\bigr|+O\left(e^{-Ax}\right).
\end{split}
\end{equation}
\end{lemma}

\begin{proof}[Proof of Lemma \ref{lem:main}]
Without loss of generality, we can assume that $g_1$ is a non-zero element since $\|\underline{g} \| \ne 0$ implies that at least one of $g_j$'s is a non-zero element.

Obviously, $\Delta_1(z)\ll e^{C|z|}$ for some positive constant $C$ depending on $U$. Let $\sigma_1$ and $\sigma_2$ be real numbers with $1/2 < \sigma_1 < \sigma_2 < 1$ such that the vertical strip $\sigma_1 < \Re s < \sigma_2$ contains the simply connected smooth Jordan domain $U$. Then for sufficiently small $\eta=\eta(U)>0$ and for all complex $z$ with $|\arg(-z)|\leq \eta$ we have $|e^{\sigma_2 z}\Delta_1(z)|\ll 1$ form the definitions of $U$ and $\sigma_2$. Furthermore, $\Delta_1$ is not identically zero since otherwise for every $k \in {\mathbb{N}}$ we have 
$$
0 = \Delta_1^{(k)}(0) = \iint_U (-s)^k \overline{g_1(s)}d\sigma dt, 
$$
which implies that $g_1$ is orthogonal to all polynomials in $B^2(U)$, however, it contradicts to $(c)$ of Lemma \ref{lem:Ber} and the assumption that $\|g_1\| \ne 0$. Hence, by using \cite[Lemma 3]{KK}, we can find a real sequence $x_k$ tending to infinity such that
\[
|\Delta_1(x_k)|\gg e^{-\sigma_2 x_k}.
\]

Fix $k$ and put $x=x_k$. Hence, by using Corollary \ref{cor:Delta}, we can see that for every $A>0$ and $x=x_k$, there exist sequences $B_1>\cdots>B_m>0$, $x^{(0)}_0=x, x^{(1)}_0,\ldots, x^{(m)}_0$ and intervals $I_j\subset[x,x+1]$ of length $|I_j|\geq B_j x^{-2j}$ such that $x^{(j)}_0\in I_j$, $I_{j+1}\subset I_j$, and for all $t\in I_{j}$, the inequalities (\ref{eq:delxj}) holds. Now let $I_m := [y,y+B_my^{-2m}]\subset[x,x+1]$. Since $I_m \subset I_j$ for every $j=1,2,\ldots,m$, the inequalities (\ref{eq:delxj}) holds also for all $t\in I_m$. In particular, since $x^{(0)}_0 = x$, for $t\in I_m$ one has
\begin{equation}\label{eq:lower1}
\bigl|\Delta_1(t)\bigr| \geq \frac{1}{2} \bigl|\Delta_1(x^{(0)}_0)\bigr|\gg e^{-\sigma_2 x}.
\end{equation}
Moreover, for every $j=1,2,\ldots,m$ we have
\begin{equation}\label{eq:upper1}
\bigl|\Delta_j(t)\bigr|\ll e^{-\sigma_1 x}, \qquad t\in [x,x+1].
\end{equation}

We denote by ${\sum_n}^*$ the sum over integers $n+\alpha \in [e^y, e^{y+B_my^{-2m}}]$ in order to obtain $\log(n+\alpha)\in I_m$.

First we consider the following sum
$$
S_1 (x) := {\sum_n}^* \sum_{j=1}^m \bigl|\Delta_j(\log(n+\alpha))\bigr|^2 .
$$
Obviously, it holds that
$$
e^{y+y^{-2m}} - e^y = e^y \bigl( e^{y^{-2m}} -1 \bigr) = \frac{e^y}{y^{2m}} \sum_{n=0}^\infty y^{-2mn}
\gg \frac{e^y}{y^{2m}} .
$$
Let $A>0$ be sufficiently large. Then by using (\ref{eq:delxj}), (\ref{eq:lower1}), $x \le y \le x+1$ and the formula above, we have
\begin{align*}
S_1(x) 
&\gg {\sum_n}^* \sum_{j=1}^m\left( \bigl|\Delta_j(x^{j}_0)\bigr|^2 
+ \bigl|\Delta_j(x^{j}_0)\bigr|O(e^{-Ax}) + O(e^{-2Ax})\right)\\
&\gg {\sum_n}^* \sum_{j=1}^m \Bigl( \bigl|\Delta_j(x^{j}_0)\bigl|^2 + O(e^{-Ax}) \Bigr)
\gg {\sum_n}^* \Biggl( \sum_{j=1}^m \bigl|\Delta_j(x^{j}_0)\bigr| \Biggl)^2 \\
&\gg {\sum_n}^* e^{-\sigma_2 x}\sum_{j=1}^m \bigl| \Delta_j(x^{j}_0) \bigr| 
\gg \frac{e^{x(1-\sigma_2)}}{x^{2m}}\sum_{j=1}^m \bigl|\Delta_j(x^{j}_0)\bigr| .
\end{align*}

Since $\lambda_k \ne\lambda_l$ for $k \ne l$ from the assumption of Theorem \ref{th:main}, then it is easy to prove that for any $1\leq k \ne l \le m$, one has
\[
\phi_{k,l}(t):=\sum_{n\leq t} e((\lambda_k-\lambda_l)n) \ll 1.
\]
Similarly to (\ref{eq:upper1}), one can easily get the estimation
\[
\frac{d}{du}\Delta_j(\log u) = \frac{1}{u}\Delta'_j(\log u) \ll u^{-1-\sigma_1} .
\]
From $\overline{\Delta_j(\log u)}=\overline{\langle u^{-s},g_j(s)\rangle} = \langle u^{-\overline{s}},\overline{g_j(s)}\rangle$, we obtain
\begin{align*}
\frac{d}{du}\overline{\Delta_j(\log u)} = \frac{1}{u}\iint_U -\overline{s}u^{-\overline{s}} 
g_j(s) d\sigma dt = \frac{1}{u}\overline{\Delta'_j(\log u)}\ll u^{-1-\sigma_1}.
\end{align*}
Hence, using partial summation, we have
\begin{align*}
&\sum_{X_1\leq n\leq X_2} \sum_{1\le k \ne l \le m} 
e((\lambda_k-\lambda_l)n)\Delta_k(\log(n+\alpha))\overline{\Delta_l(\log(n+\alpha))}\\
&\qquad = \sum_{1\le k \ne l \le m} \int_{X_1}^{X_2}
\Delta_k(\log(u+\alpha))\overline{\Delta_l(\log(u+\alpha))}d\phi_{k,l}(u)\\
&\qquad \ll X_1^{-2\sigma_1} + \sum_{1\le k \ne l \le m} \int_{X_1}^{X_2}
\left|\left(\Delta_k(\log(u+\alpha))\overline{\Delta_l(\log(u+\alpha))}\right)'\right|du\\
&\qquad \ll X_1^{-2\sigma_1}+\int_{X_1}^{X_2}\frac{du}{u^{1+2\sigma_1}}\ll X_1^{-2\sigma_1}
\end{align*}
for sufficiently large $X_2 > X_1 >0$. Thus we obtain
\begin{align*}
S_2(x) & := \sum_{1\leq k\ne l\leq m}{\sum_n}^* 
e((\lambda_l-\lambda_k)n)\Delta_k(\log(n+\alpha))\overline{\Delta_l(\log(n+\alpha))} 
\ll e^{-2\sigma_1 x}.
\end{align*}
We can easily see that
\begin{align*}
S(x) &:={\sum_n}^* \bigl|e(\lambda_1 n)\Delta_1(\log(n+\alpha))+\cdots+e(\lambda_m n)\Delta_m(\log(n+\alpha))\bigr|^2\\
&=S_1(x) + S_2 (x) \gg \frac{e^{x(1-\sigma_2)}}{x^{2m}}\sum_{j=1}^m \bigl|\Delta_j(x^{j}_0)\bigr| + O\left(e^{-2\sigma_1 x}\right)
\end{align*}
when $A$ is sufficiently large. On the other hand, one has
\begin{align*}
S(x)&\ll {\sum_n}^* \Biggl| \sum_{j=1}^m e(\lambda_j n)\Delta_j(\log(n+\alpha)) \Biggr|
\sum_{j=1}^m \bigl| \Delta_j(\log(n+\alpha)) \bigr|\\
&\ll {\sum_n}^* \Biggl| \sum_{j=1}^m e(\lambda_j n) \Delta_j(\log(n+\alpha)) \Biggr|
\sum_{j=1}^m \bigl|\Delta_j(x^{(j)}_0)\bigr|+O(e^{-(A+\sigma_1-1)x}).
\end{align*}
Hence, dividing the last inequalities by $\sum_{j=1}^m |\Delta_j(x^{(j)}_0)|$, we have
\[
{\sum_n}^* \Biggl| \sum_{j=1}^m e(\lambda_j n)\Delta_j(\log(n+\alpha)) \Biggr| 
\gg \frac{e^{x(1-\sigma_2)}}{x^{2m}},
\]
since $2\sigma_1-\sigma_2>0$. Thus, the last inequality implies Lemma \ref{lem:main}.
\end{proof}

\begin{proof}[Proof of Lemma \ref{lem:den1}]
We put
\begin{align*}
&v_n (s,\omega(n);\alpha,\lambda_j) := \frac{e(\lambda_j n)\omega(n)}{(n+\alpha)^{s}},\qquad 
\omega(n) \in \gamma,\\
&\underline{v_n} (s,\omega(n)) := 
\bigl( v_n(s,\omega(n);\alpha,\lambda_1),\ldots, v_n (s,\omega(n);\alpha,\lambda_m) \bigr).
\end{align*}
Recall $U$ be a simply connected smooth Jordan domain such that $\overline{U}\subset D$. Then the set of convergent series
\[
\biggl\{ \sum_n \underline{v_n} (s,\omega(n)):\omega\in\Omega \biggr\}
\]
is dense in the space ${\mathbb{B}}^m$ by Lemmas \ref{lem:complexHilbert} and \ref{lem:main}. Now we extend this result to the space $H(U)^m$. Let $\varepsilon >0$ and $(h_1 (s), \ldots , h_m(s)) \in H(U)^m$. From Lemma \ref{lem:main} there exists a sequence ${\beta(n) \in \gamma}$ such that $\sum_{n=0}^\infty \underline{v_n} (s,\beta(n))$ converges on $U$ in the topology of ${\mathbb{B}}^m$. This convergence is uniform on every compact subsets $\mathcal{K}_1, \ldots , \mathcal{K}_m \subset U$ by $(a)$ of Lemma \ref{lem:Ber} (see also \cite[Lemma 7]{MLerch}). Thus we can find $b(n) \in \gamma$ and $M \in {\mathbb{N}}$ satisfying
\begin{align*}
&\max_{1 \le j \le m} \max_{s \in \mathcal{K}_j} \Biggl| \sum_{n=0}^M 
v_n(s,b(n);\alpha,\lambda_j) - h_j (s) \Biggr| < \frac{\varepsilon}{2}, \\
&\max_{1 \le j \le m} \max_{s \in \mathcal{K}_j} \Biggl| \sum_{n>M} 
v_n(s,b(n);\alpha,\lambda_j) \Biggr| < \frac{\varepsilon}{2}
\end{align*}
from $(a)$ of Lemma \ref{lem:Ber} and Lemma \ref{lem:main}. The inequality above and the assumption $\overline{U} \subset D$ implies Lemma \ref{lem:den1}. 
\end{proof}

\bibliographystyle{ams}

\end{document}